\documentclass[11pt,a4paper]{amsart}
\usepackage{amsfonts}
\usepackage{amsthm}
\usepackage{amsmath}
\usepackage{amscd}
\usepackage[latin2]{inputenc}
\usepackage{t1enc}
\usepackage[mathscr]{eucal}
\usepackage{indentfirst}
\usepackage{graphicx}
\usepackage{graphics}
\usepackage{pict2e}
\usepackage{epic}
\numberwithin{equation}{section}
\usepackage[margin=2.9cm]{geometry}
\usepackage{epstopdf} 
\usepackage{verbatim}
\usepackage{amssymb}

\usepackage{hyperref}
\AtBeginDocument{}

\theoremstyle{plain}
\newtheorem{Th}{Theorem}[section]
\newtheorem{Lemma}[Th]{Lemma}
\newtheorem{Cor}[Th]{Corollary}
\newtheorem{Prop}[Th]{Proposition}

\theoremstyle{definition}
\newtheorem{Def}[Th]{Definition}

\newtheorem{Rem}[Th]{Remark}
\newtheorem{?}[Th]{Problem}

\def\al{\alpha}

\def\d{\text{d}}
\def\w{\wedge}
\def\R{\mathbb{R}}
\def\C{\mathbb{C}}
\def\Q{\mathbb{H}}

\def\Lm{\Lambda}
\def\lm{\lambda}
\def\om{\omega}
\def\Om{\Omega}
\def\vp{\varphi}

\def\y{\\[2pt]}

\def\Spin{\mathit{Spin}\,}
\def\ip{\raise1pt\hbox{\large$\lrcorner$}\>}
\def\pd#1{\frac\partial{\partial#1}}

\DeclareMathOperator{\ds}{ds}
\DeclareMathOperator{\vol}{vol}

\begin{document}
	
\title{$S^1$-quotient of $Spin(7)$-structures }

\author[U. Fowdar]{Udhav Fowdar}
	
\address{University College London \\ Department of Mathematics \\
Gower Street \\
WC1E  6BT\\
London \\
UK}
	
\email{udhav.fowdar.12@ucl.ac.uk}
		
%\subjclass[2010]{Primary: 05C??. Secondary: 05C??}
	

\begin{abstract} 
If a $Spin(7)$-manifold $N^8$ admits a free $S^1$ action preserving the fundamental $4$-form then the quotient space $M^7$ is naturally endowed with a $G_2$-structure. We derive equations relating the intrinsic torsion of the $Spin(7)$-structure to that of the $G_2$-structure together with the additional data of a Higgs field and the curvature of the $S^1$-bundle; this can be interpreted as a Gibbons-Hawking type ansatz for $Spin(7)$-structures. In particular, we show that if $N$ is a $Spin(7)$ manifold then $M$ cannot have holonomy contained in $G_2$ unless the $N$ is in fact a Calabi-Yau $4$-fold and $M$ is the product of a Calabi-Yau $3$-fold and an interval. By inverting this construction we give examples of $SU(4)$ holonomy metrics starting from torsion free $SU(3)$-structures. We also derive a new formula for the Ricci curvature of $Spin(7)$-structures in terms of the torsion forms. We then describe this $S^1$-quotient construction in detail on the Bryant-Salamon $Spin(7)$-metric on the spinor bundle of $S^4$ and on flat $\R^8$.
\end{abstract}
	
\maketitle
\section{Introduction}

In $1955$, Berger classified the possible holonomy groups of irreducible, nonsymmetric, simply connected Riemannian manifolds \cite{Berger}. The classification included the two exceptional cases of holonomy groups: $G_2$ and $Spin(7)$, of which no examples were known at the time. It is only in $1987$ that Bryant proved the existence of local examples in \cite{Bryant1987} and subsequently explicit complete non-compact examples were constructed by Bryant and Salamon in \cite{Bryant1989}. There are by now many known examples of holonomy $G_2$ and $Spin(7)$ metrics cf. \cite{Joycebook}, \cite{Brandhuber01}, \cite{Bazaikin2013}, \cite{Foscolo2015a}, \cite{Foscolo2019}, yet very few explicitly known ones. In \cite{Apostolov2003}, Apostolov and Salamon studied the $S^1$-reduction of $G_2$-manifolds and investigated the situation when the quotient is a K\"ahler manifold. By inverting their construction, they were able to give several local examples of holonomy $G_2$ metrics starting from a K\"ahler $3$-fold with additional data. Motivated by their work, in this article we shall carry out the analogous construction in the $Spin(7)$ setting, but more generally we shall look at $S^1$-invariant $Spin(7)$-structures which are not necessarily torsion free. The situation when $N$ is a $Spin(7)$ manifold has also been studied by Foscolo in \cite{Foscolo2019}. One motivation for studying the non-torsion free cases lies in the fact that they also have interesting geometric properties, for instance, balanced $Spin(7)$-structures admit harmonic spinors \cite{Ivanovspin7} and compact locally conformally parallel are fibred by nearly parallel $G_2$ manifolds \cite{Ivanovetal05}. A further motivation is that $Spin(7)$-structures have only two torsion classes and thus have only four types whereas $G_2$-structures have four classes, thus allowing for a more refined decomposition of the $Spin(7)$ torsion classes. The outline for the rest of this article is as follows.

In section $2$ we give a brief introduction to $G_2$ and $Spin(7)$-structures and set up some notation. The reader will find proofs of the mentioned facts in the standard references \cite{Bryant1987}, \cite{Salamon1989} and \cite{Joycebook}.

In section $3$ we describe the quotient of $Spin(7)$-structures which are invariant under a free circle action. The foundational result is proposition \ref{keytheorem}, which gives explicit expressions relating the torsion of the $Spin(7)$-structure on the 8-manifold $N$ to the torsion of the quotient $G_2$-structure on $M$ together with a positive function $s$ and the curvature of the $S^1$ bundle. The key observation is that this construction is reversible. In the subsequent subsections we specialise to the three cases when the $Spin(7)$-structure is torsion free, locally conformally parallel and balanced. In the torsion free situation we show that quotient manifold cannot have holonomy \textit{equal to} $G_2$ unless $N$ is a Calabi-Yau $4$-fold and $M$ is the Riemannian product of a  Calabi-Yau $3$-fold and a circle. We also give explicit expressions for the $SU(4)$-structure in terms of the data on the quotient manifold, see Theorem \ref{torsionfreequotient}. In the locally conformally parallel situation, we show that $M$ has vanishing $\Lm^3_{27}$ torsion component and furthermore, if the $\Lm^3_1$ torsion component is non-zero then $N=M\times S^1$, see Theorem \ref{locallyconformallyparallelquotient}. In the balanced situation, we show that the existence of an invariant $Spin(7)$-structure is equivalent to the existence of a suitable section of $\Lm^2_{14}$ of the quotient space, see Theorem \ref{balancedquotient}. We provide several examples to illustrate each case.

In section $4$ we derive formulae for the Ricci and scalar curvatures of $Spin(7)$-structures in terms of the torsion forms \`a la Bryant cf. \cite{Bryant06someremarks}, see Proposition \ref{ricciandscal}. As a corollary, under our free $S^1$ action hypothesis, we show that the $\Lm^2_7$ component of the curvature form corresponds to the mean curvature vector of the circle fibres. 

In the last two sections we demonstrate how our construction can be applied to the Bryant-Salamon $Spin(7)$-structure on the (negative) spinor bundle of $S^4$ and on the flat $Spin(7)$-structure on $\R^8.$ In the former case the quotient space is the anti-self-dual bundle of $S^4$ and in the latter it is the cone on $\C \mathbb{P}^3.$ We interpret the quotient of the spinor bundle as a fibrewise reverse Gibbons-Hawking ansatz. In both case we also study the $SU(3)$-structure on the link $\C \mathbb{P}^3$.
\section*{Acknowledgements}
The author would like to thank his supervisors Jason Lotay and Simon Salamon for their constant support and many helpful discussions that led to this article. This work was supported by the Engineering and Physical Sciences Research Council [EP/L015234/1]. The EPSRC Centre for Doctoral Training in Geometry and Number Theory (The London School of Geometry and Number Theory), University College London. 
\section{Preliminaries}
A $G_2$-structure on a $7$-manifold $M^7$ is given by a $3$-form $\vp$ that can be identified at each point $p\in M^7$ with the standard one on $\R^7$:
\begin{equation}
\vp_0=dx_{123}+dx_{145}+dx_{167}+dx_{246}-dx_{257}-dx_{347}-dx_{356} \end{equation}
where $dx_{ijk}$ denotes $dx_i \w dx_j \w dx_k$.
More abstractly it can equivalently be defined as a reduction of the structure group of the frame bundle of $M$ from $GL(7,\R)$ to $G_2$, but we shall use the former more concrete definition. The reason for this nomenclature is the fact that the subgroup of $GL(7,\R)$ which stabilises $\vp_0$ is isomorphic to the Lie group $G_2$. Since $G_2$ is a subgroup of $SO(7)$ \cite{Bryant1987} it follows that $\vp$ defines a Riemannian metric $g_\vp$ and volume form $vol_\vp$ on $M^7$. Explicitly these are given by
\[\frac{1}{6}\ \iota_{X}\vp\w \iota_{Y}\vp \w \vp = g_\vp(X,Y) \ vol_\vp.  \]
Thus $\vp$ also defines a Hodge star operator $*_\vp.$
It is known that a $7$-manifold admits a $G_2$-structure if and only if its first and second Stiefel-Whitney classes vanish \cite{LMspin} so there is a plethora of examples. One of the main motivations for studying this structure is that if $\vp$ is parallel with respect to the Levi-Civita connection $\nabla^{g_\vp}$ (which is a first order condition) then it has holonomy contained in $G_2$ and the metric is Ricci-flat. Such a manifold is called a $G_2$-manifold. Note that in contrast the Ricci-flat system of equations are second order. The fact that $\vp$ is parallel implies the reduction of the holonomy group from $SO(7)$ to (a subgroup of) $G_2$ and conversely, a holonomy $G_2$ metric implies the existence of such a $3$-form.  A useful alternative way to verify the parallel condition is given by the following theorem. 
\begin{Th}[\cite{Fernandez1982}]
$\nabla^{g_\vp}\vp=0$ if and only if $d\vp=0$ and $d*_\vp \vp=0.$
\end{Th}
The failure of the reduction of the holonomy group to $G_2$ is measured by the intrinsic torsion. Abstractly, given a general $H$-structure for a subgroup $H \subset O(n)$ the intrinsic torsion is defined as a section of the associated bundle to $\R^n \otimes\mathfrak{h}^\perp$ where $\mathfrak{so}(n)=\mathfrak{h}\oplus\mathfrak{h}^\perp$ and $\perp$ denotes the orthogonal complement with respect to the Killing form. We shall only give a brief description here but more details can be found in \cite{Salamon1989} and \cite{Bryant06someremarks}. The space of differential forms on $M^7$ can be decomposed as $G_2$-modules as follows:
\begin{align*}
\Lm^1 &= \Lm^1_7\\
\Lm^2 &= \Lm^2_7 \oplus \Lm^2_{14}\\
\Lm^3 &= \Lm^3_1 \oplus \Lm^3_7 \oplus \Lm^3_{27}
\end{align*}
where the subscript denotes the dimension of the irreducible module. Using the Hodge star operator we get the corresponding splitting for $\Lm^4$, $\Lm^5$ and $\Lm^6.$ 
The intrinsic torsion is given by $\dim(\R^7 \otimes \mathfrak{g}_2^\perp)=49$ equations and can be described using the equations
\begin{gather}
d\vp=\tau_0*_\vp\vp +3\ \tau_1\w\vp +*_\vp\tau_{3} \label{g2torsion1}\\
d*_\vp\vp= 4\ \tau_1 \w *_\vp\vp+\tau_2\w\vp\label{g2torsion2}
\end{gather}
where $\tau_0 \in \Om^0$, $\tau_1\in\Om^1_7$, $\tau_2 \in \Om^2_{14}$ and $\tau_{3} \in \Om^4_{27}.$ Here we are denoting by $\Om^i_j$ the space of smooth sections of $\Lm^i_j.$ The fact that $\tau_1$ arises in both equations can be proved using the following.
\begin{Lemma}[\cite{Bryant06someremarks}]\label{usefulidentities}
Given $\al \in \Lm^1_7(M)$ and $\beta \in \Lm^2_7(M)$ we have
\begin{enumerate}
	\item $2*_\vp(\beta \w *_\vp \vp) \w *_\vp \vp = 3 \beta \w \vp$
	\item $*_\vp\al =-\frac{1}{4} *_\vp(\al \w \vp)\w\vp=\frac{1}{3}*_\vp(\al \w *_\vp \vp)\w *_\vp \vp.$
\end{enumerate}
\end{Lemma}
In contrast to the non-torsion free case, manifolds with holonomy \textit{equal to} $G_2$ are much harder to find. $G_2$-structures for which $\vp$ is (co-)closed are usually called (co-)calibrated. Another notion we shall need is that of a $G_2$-instanton. 
\begin{Def}\label{instanton}
	A $G_2$-instanton on a $G_2$ manifold $(M^7,\vp)$ is a connection $1$-form $A$ on a principal $G$-bundle whose curvature form $F_A$ satisfies
	\[F_A \w *_\vp\vp=0 ,\]
	Equivalently, $F_A$ belongs to $\Om^2_{14}$, as an $\mathrm{ad}(P)$-valued $2$-form on $M^7$.
\end{Def}
Instantons are solutions of the Yang-Mills equations and as such play an important role in studying topological properties of $M^7$. There is a similar geometric structure to $G_2$-structures in dimension eight, again related to exceptional holonomy.\\
A $Spin(7)$-structure on an $8$-manifold $N^8$ is given by a $4$-form $\Phi$ that can be identified at each point $q\in N^8$ with the standard one on $\R^8$:
\begin{equation}\label{equationspin7}
\Phi_0=dx_0 \w \vp_0 + *_{\vp_0} \vp_0 \end{equation}
where we have augmented the $G_2$ module $\R^7$ by $\R$ with coordinate $x_0.$ 
The subgroup of $GL(8,\R)$ which stabilises $\Phi_0$ is isomorphic to $Spin(7)$ cf. \cite{Bryant1989, Salamon1989}.
From this definition it is clear that $G_2$ arises as a subgroup of $Spin(7).$ Since $Spin(7)$ is a subgroup of $SO(8)$ it follows that $\Phi$ defines a metric $g_\Phi$, volume form $vol_\Phi$ and Hodge star $*_\Phi.$ Explicitly the volume form is given by
\[vol_\Phi=\frac{1}{14}\ \Phi \w \Phi \]
but the expression for $g_\Phi$ is much more complicated than in the $G_2$ case cf. \cite[section $4.3$]{spirothesis}. An $8$-manifold admits a $Spin(7)$-structure if and only, if in addition to having zero first and second Stiefel-Whitney classes, either of the following holds
\[p_1(N)^2-4p_2(N)\pm 8 \chi(N)=0\]
cf. \cite{gray1970, LMspin}, noting that the `$8$' factor is accidentally omitted in the former. If $\Phi$ is parallel with respect to the Levi-Civita connection $\nabla^{g_\Phi}$ then the metric $g_\Phi$ has holonomy contained in $Spin(7)$ and the metric is Ricci-flat. Such a manifold is called a $Spin(7)$-manifold. Just as in the $G_2$ situation we have the following alternative formulation of the torsion free condition.
\begin{Th}\cite{Fernandez1986}\label{spin7thm}
$\nabla^{g_\Phi}\Phi=0$ if and only if $d\Phi=0.$
\end{Th}
The space of differential forms on $N^8$ can be decomposed as $Spin(7)$-modules as follows:
\begin{align*}
\Lm^1 &= \Lm^1_8\\
\Lm^2 &= \Lm^2_7 \oplus \Lm^2_{21}\\
\Lm^3 &= \Lm^3_8 \oplus \Lm^3_{48}\\
\Lm^4 &= \Lm^4_1 \oplus \Lm^4_7 \oplus \Lm^4_{27} \oplus \Lm^4_{35}.
\end{align*}
We shall write $\Lm^i_j(M^7)$ or $\Lm^i_j(N^8)$ if there is any possible ambiguity. There is also an injection map $i:S^2 \hookrightarrow \Lm^4$ which restricts to an isomorphism of $Spin(7)$-modules
\begin{align*}
\textup{\textsf{i}}:\langle g_{\Phi}\rangle  \oplus S^2_0 &\to \Lm^4_1 \oplus \Lm^4_{35}\\
a \circ b &\mapsto a \w *_\Phi (b \w \Phi)+b \w *_\Phi (a \w \Phi)
\end{align*}
where $S^2_0$ denotes the space of traceless symmetric $(0,2)$-tensors. Note that $\textup{\textsf{i}}(g_\Phi)=8 \Phi$. We denote by $\textup{\textsf{j}}$ the inverse map extended to $\Lm^4$ as the zero map on $\Lm^4_7 \oplus \Lm^4_{27}$. 
Similarly the intrinsic torsion is given by $\dim(\R^8\otimes \mathfrak{spin}(7)^\perp)=56$ equations and is completely determined by the exterior derivative of $\Phi$ in view of  Theorem \ref{spin7thm}. This can be written as
\begin{equation}\label{spin7torsion}
d\Phi=T^1_8 \w \Phi + T^5_{48}.
\end{equation}
where $T^5_{48}$ is defined by the condition $*_\Phi T^5_{48} \w \Phi =0.$ If $T^1_8$ vanishes the $Spin(7)$- structure is called balanced, if $T^5_{48}$ vanishes it is locally conformally parallel and if both are zero then it is torsion free.

In this article we shall often use the suggestive notation $\kappa^l_m$ for an $l$-form to mean that $\kappa^l_m \in \Om^l_m$ or write $(\kappa)^l_m$ for the $\Om^l_m$-component of an $l$-form $\kappa.$ Having set up our convention we now proceed to describe the $S^1$-reduction of $Spin(7)$-structures.
\section{The quotient construction}\label{quotientconstruction}

Given an $8$-manifold $N^8$ endowed with a $Spin(7)$-structure $\Phi$ which is invariant under a free circle action generated by a vector field $X$ the quotient manifold $M^7$ inherits a natural $G_2$-structure $\vp:=\iota_X \Phi.$ We can write the $Spin(7)$ form as
\begin{equation}\Phi = \eta \w \vp + s^{4/3} *_\vp \vp\label{basicequation} \end{equation}
where $s=\|X\|^{-1}_\Phi$, $\eta(\cdot)= s^2g_\Phi(X,\cdot)$ and $*_\vp$ is the Hodge star induced by $\vp$ on $M$. The proof for this expression is analogous to that of lemma \ref{hodgestarlemma} below. The assumption that the action is free i.e. $X$ is nowhere vanishing implies that $s$ is a well-defined strictly positive function. The metrics and volume forms of $M$ and $N$ are related by
\begin{gather}
g_\Phi=s^{-2}\eta^2 + s^{2/3}g_\vp\label{metric} \\
vol_\Phi=s^{4/3}\eta \w vol_\vp\label{volume} .
\end{gather}
In this setup $\eta$ can be viewed as a connection $1$-form on the $S^1$-bundle $N$ over $M$ and $d\eta$ is its curvature, which by Chern-Weil theory defines a section in $\Om^2(M,\mathbb{Z})$. We denote by $(d\eta)^2_7$ and $(d\eta)^2_{14}$ its two components.
Under the inclusion $G_2 \hookrightarrow Spin(7)$ we may decompose the torsion forms of (\ref{spin7torsion}) further as 
\[T^1_8=f\cdot \eta+T^1_7\]
\[T^5_{48}=T^5_{7}+T^5_{14}+\eta \w (T^4_{7}+T^4_{27}) \]
where $f$ is (the pullback of) a function on $M^7$ and all the differential forms on the right hand side, aside from $\eta$, are basic. Note that $56=8+48=(1+7)+(7+14+7+27)=49+7$ where $56$ and $49$ are the dimensions of the space of intrinsic torsions of $Spin(7)$ and $G_2$ structures. This simple dimension count confirms the absence of any $T^4_1$ term. Moreover this says that the intrinsic torsion of $\Phi$ is determined by that of $\vp$ together with a section of a rank $7$ vector bundle. In order to relate the intrinsic torsion of the $Spin(7)$-structure to that of the $G_2$-structure we first need to relate their Hodge star operators.
\begin{Lemma}\label{hodgestarlemma}
Given $\al\in\Lm^2_{7}(Y)$, $\beta\in \Lm^2_{14}(Y)$, $\gamma \in \Lm^1_7(Y)$ and using the same notation for their pullbacks to $N^8$ we have 
\begin{enumerate}
\item $*_\Phi(\al \w \vp)=-2s^{-2}\eta\w \al$
\item $*_\Phi(\beta \w \vp)=s^{-2}\eta\w \beta$
\item $*_{\Phi}\gamma=-s^{2/3}\eta \w *_\vp \gamma$
\item $*_{\Phi}\eta=s^{10/3}\vol_\vp$
\item $*_\Phi(\eta \w \al)=\frac{1}{2}s^2 \al\w\vp$
\item $*_\Phi(\eta\w\beta)=-s^2\beta\w\vp$
\item $*_\Phi(\eta\w\gamma)=s^{8/3}*_\vp \gamma$
\end{enumerate}
\end{Lemma}
\begin{proof}
This is a straightforward computation using \ref{metric}, \ref{volume} and the characterisation of $\Lm^2_7$ and $\Lm^2_{14}$ as having eigenvalues $+2$ and $-1$ under wedging with $\vp$ and taking the Hodge star \cite{Bryant1987}. We prove ($1$) as an example. Since we only need to prove the above formula holds at each point we may pick coordinates at a point $q\in N$ such that $\eta=dx_0$ and $\vp=\vp_0$. For any given $\vartheta \in \Om^2(Y)$ we then have
\[*_\Phi(\vartheta \w \vp)=-s^{-2} \eta \w *_\vp (\vartheta \w \vp) .\]
If $\vartheta=\al$ from \cite{Bryant1987} we have $*_\vp (\al \w \vp)=2\al$, which completes the proof of $(1).$ 
\end{proof}
\begin{Prop}\label{keytheorem}The intrinsic torsion of the $Spin(7)$-structure and $G_2$-structure are related by
\begin{enumerate}
\item $f=-s^{-4/3}\tau_0$
\item $7 T^1_7=24 \tau_1+3s^{-4/3}d(s^{4/3})+2s^{-4/3}*_{\vp}((d\eta)^2_7\w*_\vp\vp)$
\item $7 T^5_7=4(d\eta)^2_7\w\vp+4d(s^{4/3})\w *_\vp \vp + 4s^{4/3}\tau_1\w*_\vp\vp$
\item $T^5_{14}=(d\eta)^2_{14}\w\vp+s^{4/3}\tau_2\w\vp$
\item $T^4_{27}=-*_\vp \tau_{3}$
\item $T^4_7$ and $T^5_7$ are $G_2$-equivalent up to a factor of $s^{-4/3}$; explicitly, the composition
$$L:\Lm^5_7\xrightarrow{*} \Lm^2_7 \xrightarrow{\w*\vp} \Lm^6_7 \xrightarrow{*} \Lm^1_7 \xrightarrow{\w\vp} \Lm^4_7$$
is a bundle isomorphism and $L(7T^5_7)=4s^{-4/3}T^4_7$. 
\end{enumerate}
Moreover the occurrence of $\tau_1$ in both $(2)$ and $(3)$ shows that
\begin{equation}\label{curvature^2_7}
T^5_7-\frac{1}{6}s^{4/3}T^1_7\w*_\vp\vp =\frac{1}{2}(d(s^{4/3})\w *_\vp\vp + (d\eta)^2_7\w \vp) 
\end{equation}
and
\begin{equation}\label{tau1}
3\tau_1\w *_\vp \vp=T^1_7\w*_\vp\vp -\frac{3}{4}s^{-4/3}T^5_7;
\end{equation}
in other words any one of the three $7$-dimensional $Spin(7)$ torsion component determines the other two.
\end{Prop}
\begin{proof}
Using lemmas \ref{usefulidentities} and \ref{hodgestarlemma} we compute
\begin{align*}
*_\Phi d\Phi =\ &s^{-2}\eta \w (\d\eta)^2_{14}-2s^{-2}\eta\w(d\eta)^2_7 -3s^{2/3}*_\vp(\tau^1\w\vp)-\tau_0\ s^{2/3}\vp-s^{2/3}\tau_{3}\\ &-s^{-2}*_\vp(d(s^{4/3})\w*_\vp\vp)\w \eta + s^{-2/3}\tau_2\w\eta-4s^{-2/3}\eta\w*_\vp(\tau^1\w*_\vp\vp). 
\end{align*}
It now suffices to use the identity $7 *_\Phi T^1_8= *_\Phi(d\Phi)\w\Phi$ cf. \cite{spirospin7} and compare terms.
\end{proof}
\begin{Rem}
Note that the above construction can also be extended to non-free $S^1$ actions by working on the complement of the fixed point locus. The fixed point locus then corresponds to the region where $s$ blows up. We shall in fact see an example of this below when we look at the Bryant-Salamon $Spin(7)$ metric.
\end{Rem}
Equipped with above proposition we can now proceed to studying the quotient of different types of $Spin(7)$-structures.
\subsection{The torsion free quotient}\label{tf}
\begin{Th}
Assuming $(N^8,\Phi)$ is a $Spin(7)$-manifold, the quotient $G_2$-structure $\vp$ is calibrated and the curvature is determined by
\begin{equation}\label{equ11}(d\eta)^2_7 \w*_\vp\vp=-\frac{3}{2}*_\vp d(s^{4/3}) \end{equation}  and \begin{equation}\label{equ22}(d\eta)^2_{14} =- s^{4/3}\tau_2 .\end{equation}
\end{Th}
\begin{proof}
This follows directly from proposition \ref{keytheorem}. From (1), (\ref{tau1}) and (5) we see that $\tau_0$, $\tau_1$ and $\tau_{3}$ must vanish. The curvature equations follow from (\ref{curvature^2_7}) and (4).
\end{proof}
The above equations have also been described as a Gibbons-Hawking type ansatz for $Spin(7)$-manifolds in \cite{Foscolo2019}, where the author studies `adiabatic limits' of the equations to produce new complete non-compact $Spin(7)$ manifolds. The pair (\ref{equ11}) and (\ref{equ22}) generally constitute a complicated system of PDEs. A strategy to solving this system and hence constructing  $Spin(7)$ metrics on the total space involves taking a formal limit of the equations as the size of the circle fibres tend to zero and thus, allowing for the system to degenerate to the torsion free $G_2$ equations. One then employs analytical techniques to perturb the latter equations to construct solutions to the original system. This limiting procedure of  shrinking the fibres is referred to as the `adiabatic limit' following a related strategy outlined in \cite{Donaldson2017} in the context of $K3$-fibred $G_2$ manifolds.
\begin{Rem}\text{\\ }
\begin{enumerate}
\item First we note that in our setting if $(N,\Phi)$ has holonomy \textit{equal to} $Spin(7)$ then it is necessarily non-compact. This follows essentially from the Cheeger-Gromoll splitting theorem which asserts that if $(N,\Phi)$ is compact and Ricci-flat then its universal cover is isometric to $\R^k \times P^{8-k}$ where $P$ is a simply connected Riemannian manifold and $\R^k$ carries the flat metric. Under the hypothesis that it admits a free isometric $S^1$ action it follows that $k\geq 1$ which together with Berger's classification of holonomy groups implies that (the identity component of) the holonomy group of $N$ must be a subgroup of $G_2.$
\item If the size of the circle orbits are constant i.e. $s$ is constant then $\tau_2$ is proportional to $d\eta$ so in particular $\tau_2$ is closed. But from equation ($4.35$) of \cite{Bryant06someremarks}
\[d\tau_2= \frac{1}{7}\|\tau_2\|^2_\vp + (d\tau_2)^3_{27}  \]
and hence $\tau_2=0$ i.e. $N^8=S^1 \times M^7$ is a Riemannian product.
\end{enumerate}
\end{Rem}
%Note that the former equation is equivalent to $(d\eta)^2_7\w\vp=-d(s^{4/3})\w *_\vp\vp.$
If we now further demand that $(M^7,\vp)$ is also torsion free then this forces the connection to be a $G_2$-anti-instanton i.e. $d\eta \in \Lm^2_7$, compare with definition \ref{instanton}. Since $ds$ is closed, $\nabla ds \in S^2(T^*Y)\cong\Lm^3_{1} \oplus \Lm^3_{27}$ cf. \cite{Bryant1989} \cite{Salamon1989} but we also have
$$d\eta \w *_\vp \vp = -\frac{3}{2}*_\vp d(s^{4/3})$$
i.e. $d\eta$ and $d(s^{4/3})$ are $G_2$-equivalent, therefore the two components of $\nabla ds$ are completely determined by the $\Lm^3_1$ and $\Lm^3_{27}$ components of $dd\eta=0 \in \Lm^3$. Hence $ds$ is a covariantly constant $1$-form as such $Hol \subsetneq G_2$ \cite[Theorem 4]{Bryant1989}. Thus we have proven the following.
\begin{Th}\label{torsionfreequotient}
If $(N^8,\Phi)$ is a torsion free $Spin(7)$-structure which is invariant under a free $S^1$ action such that the quotient structure has holonomy contained in $G_2$ then $M^7=Z^6 \times \R^+$ where $(Z^6,h,\om,\Om:=\Om^+ + i\Om^-)$ is a Calabi-Yau $3$-fold. Furthermore $(N^8,\Phi)$ is a Calabi-Yau $4$-fold and is given by $\Phi=\frac{1}{2} \hat{\om}^2 + Re(\hat{\Om})$ where
\begin{gather}
\hat{\om}= s^{2/3} \om + \eta \w d(s^{2/3})\label{symplectic} \\
\hat{\Om}=\Om \w (- \eta -i \frac{2}{3} s^{{5/3}}ds)\label{complex} \\
\hat{h}= s^{{2/3}}h + s^{-2}\eta^2 + (\frac{2}{3}s^{{2/3}} \ds)^2\label{metriccy}
\end{gather}
defines the $SU(4)$-structure and $s$ is the coordinate on the $\R^+$ factor. The curvature form is $d\eta=-\om$, corresponding to a $G_2$-anti-instanton, and the product $G_2$-structure is given by
\begin{gather*}
\vp=\frac{2}{3}s^{1/3}ds\w\om+\Om^+ \\
*_\vp\vp=\frac{1}{2}\om^2-\frac{2}{3}s^{1/3}ds\w\Om^-.
\end{gather*}
Moreover this construction is reversible i.e. starting from a CY $3$-fold $(Z^6,h,\om,\Om)$ we can choose a connection form $\eta$ satisfying $d\eta=-\om$ on the bundle defined by $[-\om]\in H^2(Z^6,\mathbb{Z})$ together with a positive function $s$ and thus define an irreducible CY $4$-fold $(N^8,\hat{h},\hat{\om},\hat{\Om})$ by (\ref{symplectic}) ,(\ref{complex}) and (\ref{metriccy}).
\end{Th}
The above theorem in fact recovers the so-called \textit{Calabi model space}. More precisely, given a compact CY manifold together with an ample line bundle $L$, the Calabi ansatz gives a way of defining a new CY metric on an open set of $L$. Although incomplete, the Calabi model space provides a good approximation for the asymptotic behaviour of the complete Tian-Yau metrics and has recently been employed in \cite{hsvz} to study new degenerations of hyperK\"ahler metrics on $K3$ surfaces. Observe that taking $(Z^6, h)$ to be the Riemannian product of a hyperK\"ahler metric obtained by the Gibbons-Hawking ansatz and a flat torus $\mathbb{T}^2$ we get infinitely many holomomy $SU(4)$ metrics. We now give a simple example to illustrate this construction. The metric below has also been described in \cite{MarcoFreibert2018} as a solution to the Hitchin flow starting from a $7$-nilmanifold endowed with a cocalibrated $G_2$-structure.\\
\textbf{Example.}
Consider $\mathbb{T}^6$, with coordinates $\theta_i\in [0,2\pi)$, endowed with the flat CY-structure
\begin{gather*}
\om=e^{12} +e^{34}+e^{56},\\
\Om= (e^{1}+ie^2)\w(e^{3}+ie^4)\w(e^{5}+ie^6),
\end{gather*}
where $e^i=d\theta_i$. $[-\om]\in H^2(\mathbb{T}^6,\mathbb{Z})$ defines a non-trivial $S^1$- bundle diffeomorphic to the nilmanifold $P$ with nilpotent Lie algebra $(0,0,0,0,0,0,12+34+56)$ where we are using Salamon's notation cf. \cite{Salamoncomplexstructures}.  The connection form is given by 
$$\eta=d\theta_7+\theta_2 e^1+\theta_4 e^3+\theta_6 e^5,$$ where $\theta_7$ denotes the coordinate of the $S^1$ fibre. Writing $s=r^3$ the CY metric on $P\times \R^+$ can be written as
\[\hat{h}=r^2g_{\mathbb{T}^6}+r^{-6}(d\theta_7+\theta_2 e^1+\theta_4 e^3+\theta_6 e^5)^2 + 4r^8dr^2 . \]
Using \textsc{Maple} we have been able to verify that indeed the matrix of curvature $2$-form has rank $15$ everywhere, confirming that the holonomy is \textit{equal} to $SU(4).$
If we set $\rho=\frac{2}{5}r^5$ then the metric can be written as
\[\hat{h}=\big(\frac{5}{2}\ \rho\big)^{2/5}g_{\mathbb{T}^6}+\big(\frac{5}{2}\ \rho\big)^{-6/5}(d\theta_7+\theta_2 e^1+\theta_4 e^3+\theta_6 e^5)^2 + d\rho^2 \]
and in this form we can easily show that the volume growth $\sim \rho^{8/5}$ and the curvature tensor $|\textrm{Rm}| \sim \rho^{-2}$ as $\rho \to \infty.$ This metric is in fact incomplete at the end $\rho \to 0$ and complete at the end $\rho \to \infty.$ By way of comparing with the approach in \cite{MarcoFreibert2018}, the $SU(4)$ holonomy metric can also be obtained by evolving the cocalibrated $G_2$-structure on $P$ given by
\[\vp=\eta \w \om + Re(\Om),\] 
in the notation of  Theorem \ref{torsionfreequotient}. Our approach however avoids the problem of having to solve the Hitchin flow evolution equations and moreover, it explains why one only obtains $SU(4)$ holonomy metrics rather than $Spin(7)$ ones.

As we have just seen one cannot obtain a holonomy $G_2$ metric from a $Spin(7)$ manifold via this construction. This suggests to study instead the geometric structure of the quotient calibrated $G_2$-structure. We shall do so in detail for the Bryant-Salamon $Spin(7)$-metric in section \ref{bsconequotient}.
\subsection{The locally conformally parallel quotient}
\begin{Th}\label{locallyconformallyparallelquotient}
If $(N^8,\Phi)$ is a locally conformally parallel $Spin(7)$-structure which is $S^1$-invariant then at least one of the following holds:
\begin{enumerate}
\item $N^8\simeq M^7\times S^1$ and the $G_2$-structure on $M$ has $\tau_{3}=0$ in the notation of \ref{g2torsion1}, or
\item $(M^7,\vp)$ is locally conformally calibrated i.e. $\tau_0$ and $\tau_3$ are both zero, and hence $\tau_1$ is closed.
\end{enumerate}
\end{Th}
\begin{proof}
Since $T^5_{48}=0$ we know that $T^5_7$, $T^5_{14}$,$T^4_7$ and $T^4_{27}$ all vanish. From Proposition \ref{keytheorem} it follows that $\tau_0=-s^{4/3}f$, $\tau_1=-s^{-4/3}(d(s^{4/3})+\frac{2}{3}*_\vp((d\eta)^2_7\w*_\vp \vp) )$, $\tau_2=-s^{-4/3}(d\eta)^2_{14}$ and $\tau_{27}^4=0$. From Proposition \ref{keytheorem} we also get
\[T^1_7=-3s^{-4/3}d(s^{4/3})-2s^{-4/3}*_\vp((d\eta)^2_7\w*_\vp\vp) \]
Furthermore, differentiating $d\Phi=T^1_8 \w \Phi$ we have
\[dT^1_8\w\Phi=0. \]As wedging with $\Phi$ defines an isomorphism of $\Lm^2$ and $\Lm^6$ it follows that $T^1_8$ is closed. Since $\mathcal{L}_X\Phi=0$ we have 
\[d(\iota_Xd\Phi)=0\]
and this shows that 
\[\mathcal{L}_X T^1_8 \w \Phi =d(\iota_Xd\Phi)=0.\]
Thus $f=T^1_8(X)$ is constant and if non-zero then
\[d\eta= -\frac{1}{f}dT^1_7.  \]
Since the latter is exact, the Chern class is zero and the bundle is topologically trivial i.e. $N^8 \simeq M^7\times S^1.$ Otherwise if $f=0$ then $\tau_0=0.$
\end{proof}
In \cite[Theorem B]{Ivanovetal05} Ivanov et al. prove that any compact locally conformally parallel $Spin(7)$-structure fibres over an $S^1$ and each fibre is endowed with a nearly parallel $G_2$-structure i.e. the only non-zero torsion form is $\tau_0.$ Thus, it follows from Proposition \ref{keytheorem} that one can construct many such examples by taking $N^8=M^7\times S^1$ where $M^7$ is a nearly parallel $G_2$-manifold and endow $N^8$ with the product $Spin(7)$-structure. In particular these examples cover case $(1)$ above where the $S^1$ is only acting on the second factor. We also point out that aside from the fact that the cone metric on a nearly parallel $G_2$ manifold has holonomy contained in $Spin(7)$, there exists another Einstein metric, with instead positive scalar curvature, on $(0,\pi)\times M^7$ given by the sine-cone construction:
\[g_{sc}:=dt^2+\sin(t)^2g_{M^7}. \] 
The latter metric however does not seem to have been studied in detail in the literature. The fact that $g_{sc}$ is Einstein is easily deduced since its Riemannian cone is Ricci-flat. Let us now show how situation $(2)$ can arise. The reader might find it helpful to compare the following example with section \ref{flatquotient}. \\ 
\textbf{Example.} As above let $N^8=S^7\times S^1$, where $S^7$ is given the nearly parallel $G_2$-structure induced by restricting $\Phi_0$ to $S^7 \hookrightarrow \R^8.$ The induced $G_2$-structure $\vp_{S^7}$ satisfies
\[d\vp_{S^7}=4\ *_{S^7}\vp_{S^7} \]
and defines the standard round metric on $S^7.$ Consider any free $S^1$ action, generated by a unit vector field $X$, on $S^7$ preserving $\vp_{S^7}$. We can then write
\[\vp_{S^7}=\eta \w \om + \Om^+ \text{\ \ and \ }*_{S^7}\vp_{S^7}=\frac{1}{2}\om\w\om-\eta\w\Om^- \]
cf. \cite{Apostolov2003}.
The intrinsic torsion of the quotient $G_2$-structure on $M^7=\C P^3 \times S^1$, with coordinate $\theta$ on the circle, is then given by
\begin{align*}
d\vp&= 3(-\frac{4}{3}\ d\theta )\w \vp,\\
d*_\vp\vp&=4(-\frac{4}{3}\ d\theta )\w *_\vp\vp-(\frac{2}{3}\om+d\eta)\w \vp,
\end{align*}
confirming that indeed $\tau_0$ and $\tau_3$ vanish but $\tau_1$ and $\tau_2$ do not, cf. (\ref{g2torsion1}) and (\ref{g2torsion2}).
\subsection{The balanced quotient}
Since $T^1_8=0$, from proposition \ref{keytheorem} $(1)$ we have $\tau_0=0$ and $(2)$ gives 
\begin{equation}\label{balancedspin}
\tau_1=-\frac{1}{24}(3s^{-4/3}d(s^{4/3}) +2s^{-4/3}*_\vp((d\eta)^2_7 \w *_\vp \vp) ).
\end{equation} %This imposes
%\[T^5_7=\frac{7}{2}((d\eta)^2_7\w\vp+d(s^{4/3})\w*_\vp\vp). \]
\begin{Rem}
Differentiating the balanced condition $*_\Phi(d\Phi)\w\Phi=0$ we get
\[ \|d\Phi \|^2_{\Phi}vol_\Phi=- (d*_\Phi d\Phi)\w\Phi = (\Delta_{\Phi}\Phi)\w\Phi.\]
In particular this shows that $d\Phi=0$ i.e. $\Phi$ is torsion free iff \[\Delta_\Phi\Phi \w\Phi=0 \]
which is a single scalar PDE.
\end{Rem}
It is well-known that a $Spin(7)$-structure can be equivalently characterised by the existence of a non-vanishing spinor $\psi$, instead of the $4$-form $\Phi$. Following Theorem \ref{spin7thm}, the induced metric has holonomy contained in $Spin(7)$ if and only if the spinor is parallel. From this perspective the action of the Dirac operator $D$ on the spinor was shown to be completely determined by the torsion form $T^1_8$ cf. \cite[($7.21$)]{Ivanovspin7}. As a consequence, it follows that balanced $Spin(7)$-structures are characterised by the fact that they admit harmonic spinors i.e. $D\psi=0$.

In \cite{Lucia2019} the authors construct many such examples on nilmanifolds by adopting a spinorial point of view. 
We instead here describe, via a few simple examples, a construction of balanced $Spin(7)$-structures starting from suitable $G_2$-structures. Henceforth we shall restrict to the case when $s=1$ so that (\ref{balancedspin}) can be equivalently written as
\begin{equation}(d\eta)^2_7=-4*_{\vp}(\tau_1\w *_\vp\vp).\label{curvature27} \end{equation}
\begin{Th}\label{balancedquotient}
$(N^8,\Phi)$ is a free $S^1$-invariant balanced $Spin(7)$-structure if and only if the $G_2$-structure $(M^7,\vp)$ has $\tau_0=0$ and admits a section $\lm\in\Om^2_{14}$ such that 
\[ [\lm - 4*_\vp(\tau_1 \w *_\vp\vp)] \in H^2(M,\mathbb{Z}) \]
or equivalently,
\begin{equation}\label{set} \{\kappa +4*_\vp(\tau_1 \w *_\vp\vp)\ |\ [\kappa]\in H^2(M,\mathbb{Z})  \}\cap \Om^2_{14} \neq \emptyset. \end{equation}
Moreover, the $Spin(7)$-structure on the total space can be written as
\begin{equation}\Phi=\eta \w \vp + *_\vp\vp\label{spin7connection} \end{equation}
where the connection form $\eta$ satisfies $d\eta=\lm-4*_\vp(\tau_1 \w *_\vp\vp)$ i.e.
\begin{equation*} (d\eta)^2_7=-4*_\vp(\tau_1 \w *_\vp\vp)\ \text{\ and \ } (d\eta)^2_{14}=\lm. \label{cur} \end{equation*}
\end{Th}
\begin{proof}
The \textit{if} statement is clear since given $\lm$ we can always choose a connection $\eta$ with $d\eta=\lm-4*_\vp(\tau_1 \w *_\vp\vp)$. Then define $\Phi$ by (\ref{spin7connection}). The \textit{only if} statement follows by setting $\lm=(d\eta)^2_{14}.$
\end{proof}
The reader might find such a theorem of little practical use in general. However, as we shall illustrate below via concrete examples, when $M^7$ is a nilmanifold Theorem \ref{balancedquotient} provides a systematic way of constructing balanced $Spin(7)$-structures.\\
\iffalse
\textbf{Example 1.} Let $M=\mathbb{T}^7$ with the flat coframing $\langle e^0,e^2,e^3,e^4,e^5,e^6,e^7\rangle$ and define a torsion free $G_2$-structure by
\[\vp=e^{032}-e^{045}-e^{067}-e^{346}-e^{375}-e^{247}-e^{256} .\]
The $2$-form $e^{23}-e^{45} \in \Lm^2_{14}$ defines a non-zero element of $H^{2}(\mathbb{T}^7,\mathbb{Z})$ and thus a non-trivial $S^1$-bundle. The total space $N$ is in fact diffeomorphic to the nilmanifold with nilpotent Lie algebra $(0,0,0,0,0,0,0,12+34)$. We denote a connection form $\eta$ on $N$ by $e^1$ such that $d\eta=e^{23}-e^{45}$. The induced $Spin(7)$-form can now be written in the standard form (\ref{equationspin7}) with $dx_i$ replaced by $e^i.$ From $(1)$ and $(2)$ of Proposition \ref{keytheorem}, it follows that $T^1_8=0.$\\ This simple construction can be applied to any $G_2$-structure with vanishing $\tau_0$ and $\tau_1$, by setting $s$ to be constant and choosing classes in $H^2(M,\mathbb{Z})$ that have representatives in $\Om^2_{14}$. We now consider the more general situation when $\tau_1\neq0.$ \\ \fi
\textbf{Example.} Let $M^7=B^5\times \mathbb{T}^2$, where $B$ is a nilmanifold with an orthonormal coframing given by $e^i$ for $i=0,...,4$ and satisfying 
\begin{gather*}
de^i=0, \ \  \text{for}\ i\neq 4\\
de^4=e^{02}+e^{31},
\end{gather*}
and for the flat $\mathbb{T}^2$ by $e^6$ and $e^7.$ The $G_2$-structure defined by
\[\vp=e^{137}+e^{104}+e^{162}+e^{306}+e^{324}-e^{702}-e^{746}. \]
has $\tau_0=0$. Hence from (\ref{curvature27}), to construct a balanced $Spin(7)$-structure we need to find a connection $\eta$ whose $\Lm^2_7$-curvature component satisfies
\begin{align*}
(d\eta)^2_7
&=-4*_{\vp}(\tau_1\w *_\vp\vp)\\
&=\frac{2}{3}(e^{03}+e^{12}-e^{47}).
\end{align*}
Choosing $(d\eta)^2_{14}$ to be either of following $2$-forms in $\Om^2_{14}$:
\begin{gather*}
\frac{1}{3}(e^{03}+e^{12}+2e^{47}),\\
\frac{2}{3}(2e^{12}-e^{03}+e^{47})
\end{gather*}
gives connections with curvature forms $e^{03}+e^{12}$ and $2e^{12}$ respectively, and thus we obtain two distinct balanced $Spin(7)$-structures. Denoting $\eta$ by $e^5$ the $Spin(7)$ form can once again be written in the standard form (\ref{equationspin7}). This construction shows that given a balanced $Spin(7)$-structure on an $S^1$-bundle we can modify the $\Lm^2_{14}$-component of the curvature form while keeping its $\Lm^2_7$-component, already determined by $\tau_1$, unchanged to construct a new balanced structure.

Suppose that we have fixed $d\eta=de^5=2e^{12}.$ We can now take the $S^1$-quotient with respect to the Killing vector field $e_4.$
In other words, the total space can be viewed as a different circle bundle with the new connection form $\tilde{\eta}:=e^4.$ We can repeat the above procedure with the new $G_2$-structure $\tilde{\vp}:={e_4}\ip \Phi$, explicitly given by
\[\tilde{\vp}=e^{501}+e^{523}+e^{567}+e^{026}+e^{073}-e^{127}-e^{136}, \]
which of course has $\tilde{\tau}_0=0$. Once again to construct a balanced $Spin(7)$-structure we need a connection $1$-form $\xi$ satisfying
\begin{align*}
(d\xi)^2_7
&=(d\tilde{\eta})^2_7\\
&=\frac{2}{3}(e^{02}+e^{31}-e^{57}).
\end{align*}
If we choose 
\[(d\xi)^2_{14}=(d\tilde{\eta})^2_{14}+e^{51}+2e^{26}+e^{37}  \]
then $d\xi=e^{02}+e^{31}+e^{51}+2e^{26}+e^{37}$ indeed defines an element in $H^2(\tilde{M},\mathbb{Z})$. Thus this gives yet another balanced $Spin(7)$-structure. These three examples were found in \cite{Lucia2019} denoted by $\mathcal{N}_{6,22}$, $\mathcal{N}_{6,23}$ and $\mathcal{N}_{6,24}$, by instead using the spinorial approach described above and computing the Dirac operator. 

The above examples in fact illustrate a new procedure for constructing balanced $Spin(7)$-structures on nilmanifolds: Starting from an $S^1$-invariant balanced $Spin(7)$-structure on a nilmanifold we know that the quotient $G_2$-structure $\vp$ has $\tau_0=0$. Given that the de Rham complex of the quotient nilmanifold $P^7$ is completely determined by the Chevalley-Eilenberg complex of the associated nilpotent Lie algebra, it is relatively straightforward to compute the set (\ref{set}), via say \textsc{Maple}. Thus, by choosing distinct $\lambda$s we can classify all invariant balanced $Spin(7)$-structures on \textit{different} nilmanifolds which arise as circle bundles over $(P^7,\vp)$. A general classification however appears to be quite hard. Closed $G_2$-structures on nilpotent Lie algebras, hence with $\tau_0=0$, were classified in \cite{Conti2011a}. Although a classification of $7$-dimensional nilpotent Lie algebras is known cf. \cite{GongNil7}, those admitting $G_2$-structures with only vanishing $\tau_0$ is still unknown.

Having encountered several examples of $Spin(7)$-structures
it seems worth making a brief digression from our main example and derive some curvature formulae of $Spin(7)$-structures in terms of the torsion forms, rather than the metric, that the reader might find quite practical in specific examples.
\section{Ricci and Scalar curvatures}
In this section we derive formulae for the Ricci and scalar curvatures of $Spin(7)$-structures in terms of the torsion forms. As a corollary we show that under the free $S^1$ action hypothesis and that the circle orbits have constant size, $(d\eta)^2_7$ can be interpreted as the mean curvature vector of the circle orbits.\\
Formulae for the Ricci and scalar curvatures of $G_2$-structures in terms of the torsion forms seem to have first appeared in \cite[$(4.28),(4.30)$]{Bryant06someremarks} and for the $Spin(7)$ case in \cite[$(1.5),(7.20)$]{Ivanovspin7}. The approach taken in each paper to derive the curvature formulae differ greatly. While Ivanov uses the equivalent description of $Spin(7)$-structures as corresponding to the existence of certain parallel spinors, Bryant uses a more representational theoretic argument. In \cite{Ivanovspin7}, however, it is not obvious from the Ricci formula that it is a symmetric tensor and moreover the presence of a term involving the covariant derivative of the torsion form makes explicit computations quite hard. We instead adapt the technique outline in \cite[Remark 10]{Bryant06someremarks} to the $Spin(7)$ setting and derive an alternative formula.
\begin{Prop}\label{ricciandscal}
The Ricci and scalar curvatures of a $Spin(7)$-structure $(N,\Phi)$ are given by
\begin{align*} % this is proved in 27th july notes
Ric(g_{\Phi})=&\bigg(\frac{5}{8}\ \delta T^1_8 +\frac{3}{8}\ \|T^1_8 \|^2_\Phi -\frac{2}{7}\ \|T^5_{48}\|^2_\Phi\bigg)g_{\Phi}\\
&+ \textup{\textsf{j}}\bigg(-3\ \delta(T^1_8 \w \Phi) + 4\ \delta T^5_{48} -2\  (T^1_8 \w *_\Phi T^5_{48})-\frac{9}{4} *_\Phi(T^1_8 \w \Phi)\w T^1_8) \bigg)\\
&+\frac{1}{2}\ g_{\Phi}(\cdot\ \ip *_\Phi T^5_{48},\cdot\ \ip *_\Phi T^5_{48})
\end{align*}
\[Scal(g_\Phi)=\frac{7}{2}\ \delta T^1_8+\frac{21}{8}\ \|T^1_8 \|^2_{\Phi}-\frac{1}{2}\ \|T^5_{48} \|^2_{\Phi}.\]
where $\delta:=-*_\Phi d *_\Phi$ is the codifferential of $\Phi$.
\end{Prop}
\begin{proof}
Following Bryant's argument in \cite{Bryant06someremarks} for the $G_2$ case, we first define the two $Spin(7)$-modules $V_1$ and $V_2$ by
\[(\mathfrak{gl}(8,\R)/\mathfrak{so}(7))\otimes S^k(\R^8)=V_k \oplus (\R^8 \otimes S^{k+1}(\R^8)) ,\]
where $S^k(\R^8)$ denotes the $k^{th}$ symmetric power.
We shall refer to these modules to also mean the corresponding associated vector bundles on $N$. Representing irreducible $Spin(7)$-modules by the highest weight vector we have the following decomposition:
\begin{gather*}
V_1=V_{0,0,1}\oplus V_{1,0,1},\\
V_2=V_{0,0,0}\oplus V_{1,0,0}\oplus V_{0,1,0}\oplus V_{1,1,0}\oplus V_{2,0,0}\oplus V_{0,2,0}\oplus 2V_{0,0,2}\oplus V_{1,0,2},\\
S^2(V_1)=2V_{0,0,0}\oplus V_{1,0,0}\oplus V_{0,1,0}\oplus 2V_{1,1,0}\oplus 2V_{2,0,0}\oplus V_{0,2,0}\oplus 4V_{0,0,2}\oplus 2V_{1,0,2}\oplus V_{2,0,2}.
\end{gather*}
It is known that the second order term of the scalar curvature values in the trivial component of $V_2$ of which there is only one. This is spanned by $\delta T^1_8$. The first order terms are at most quadratic in sections of $V_1$ of which there are only two components. These are just the norm squared of the torsion forms: $\|T^1_8\|^2_\Phi$ and $\|T^5_{48}\|^2_\Phi$. So the scalar curvature can be expressed in terms of these three terms and to determine the coefficients it suffices to test it on a few examples. A similar argument applies for the traceless part of the Ricci tensor. The second order terms correspond to sections of the module $V_{0,0,2}\cong S^2_0(\R^8)$  in $V_2$ and there are exactly two of those. These are spanned by the projections of $\delta(T^1_8 \w \Phi)$  and $\delta T^5_{48}$. For the first order terms, they are given by sections of the module $V_{0,0,2}$ in $S^2(V_1).$ There are in fact four of those; one quadratic in $T^1_8$, two quadratic in $T^5_{48}$ and one mixed term. All but one quadratic term in $T^5_{48}$ appear in the Ricci formula. Again to determine the coefficients it suffices to test the formula on a few examples. This can be done quite easily using \textsc{Maple}.
\end{proof}
\noindent From the results of section \ref{quotientconstruction} we have the following lemma.
\begin{Lemma}
In the $S^1$-invariant setting, $\delta T^1_8$, $\|T^1_8 \|^2_\Phi$ and $\|T^5_{48}\|^2_\Phi$  are given in terms of the data $(M^7,\vp,\eta,s)$ by
	%\begin{align}
	%T=&\ \frac{1}{6} \tau_0 s^{2/3}\vp -s^{-2}\eta \w (d\eta)^2_7-\frac{1}{2}s^{-2}\eta\w*_\vp(d(s^{4/3}) \w *_\vp \vp)\\& + s^{-2}\eta \w (d\eta)^2_{14}+s^{-2/3}\eta \w\tau_2  +s^{2/3}*_\vp(\tau_1\w\vp)\nonumber \\&+\frac{1}{3}s^{-2/3}*_\vp(*_\vp((d\eta)^2_7\w *_\vp\vp)\w\vp)-s^{2/3}\tau_3\nonumber
	%\end{align}
\begin{equation}
\delta T^1_8=\frac{1}{7} s^{-4/3}\delta_\vp(24 s^{2/3} \tau_1+4s^{-1/3}ds+2s^{-2/3}*_{\vp}((d\eta)^2_7\w*_\vp\vp))) 
\end{equation}
\begin{equation}
\|T^1_8 \|^2_\Phi=s^{-2/3}\tau_0^2+\frac{1}{49}s^{-2/3}\|24\tau_1+4s^{-1}ds+2s^{-4/3}*_\vp((d\eta)^2_7\w*_\vp\vp) \|^2_\vp
\end{equation}
	%\begin{align}
	%\|T\|^2_\Phi=& \frac{7}{36}\tau_0^2 s^{-2/3}+s^{-2/3}\|\tau_3\|^2_\vp+s^{-4/3}\|s^{-1}(d\eta)^2_{14}+s^{1/3}\tau_2 \|^2_\vp\\ \ &+ s^{-10/3}\|(d\eta)^2_7+\frac{1}{2}*_\vp(d(s^{4/3})\w*_\vp\vp) \|^2_\vp\nonumber \\ \ &+4\|s^{2/3}\tau_1+\frac{1}{3}s^{-2/3}*_\vp((d\eta)^2_7 \w *_\vp\vp) \|^2_\vp\nonumber
	%\end{align}
\begin{align}
\|T^5_{48}\|^2_\Phi=& s^{-2/3}\|\tau_3\|^2_\vp+s^{-4/3}\|s^{-1}(d\eta)^2_{14}+s^{1/3}\tau_2 \|^2_\vp\\ \ &+ s^{-10/3}\|\frac{8}{7}(d\eta)^2_7+\frac{4}{7}*_\vp(d(s^{4/3})\w*_\vp\vp)+\frac{4}{7}s^{4/3}*_\vp(\tau_1\w*_\vp\vp) \|^2_\vp\nonumber \\ \ &+4\|\frac{3}{7}s^{2/3}\tau_1+\frac{2}{7}s^{-2/3}*_\vp((d\eta)^2_7 \w *_\vp\vp)+\frac{3}{7}s^{-2/3}d(s^{4/3}) \|^2_\vp\nonumber
\end{align}
where $\delta_\vp$ is the codifferential of $\vp$ acting on $k$-forms by $\delta_\vp = (-1)^k *_\vp d *_\vp.$ 
	
\end{Lemma}
\begin{proof}
This is a straightforward albeit long computation using the expressions for the torsion forms of the $Spin(7)$-structure from Proposition \ref{keytheorem}.
\end{proof}
Of course these formulae are far from practical to compute the scalar curvature but nonetheless in the case of Riemannian submersions they do simplify considerably.
\begin{Cor}
In the case of a Riemannian submersion i.e. $s=1$, 
\begin{align*}Scal(g_\Phi)=Scal(g_\vp)&-\frac{1}{2}\|d\eta \|^2_\vp-g_\vp((d\eta)^2_{14} ,\tau_2) +\delta_\vp(*_\vp((d\eta)^2_7 \w *_\vp \vp))\\&+4g_\vp(*_\vp\tau_1, (d\eta)^2_7 \w *_\vp \vp). \end{align*}
\end{Cor}
\begin{proof}
This follows by combining the above lemma with our formula for scalar curvature and the one in the $G_2$ case from \cite[(4.28)]{Bryant06someremarks}.
\end{proof}
\begin{Rem}
Comparing the above formula with the general formula for scalar curvatures in Riemannian submersions cf. \cite[$(9.37)$]{Besse2008}, we can geometrically interpret the anti-instanton part of the curvature form: $$*_\vp((d\eta)^2_7\w *_\vp\vp)$$ as the dual with respect to $g_\vp$ of the `mean' curvature vector of the $S^1$ fibres. 
For an immersed submanifold of codimension greater than one, the mean curvature is defined by a normal vector, rather than a scalar, cf. \cite[$(1.73)$]{Besse2008}. In our present situation the `mean' curvature of the circle fibres is determined by a vector in the rank $7$ normal bundle, which we can identify with $(d\eta)^2_7$. Therefore it vanishes if and only if the circles are geodesics. Of course the word `mean' here is redundant as our submanifolds are only one dimensional.
\end{Rem}
We now turn to our main example namely the $S^1$ quotient of the Bryant-Salamon metric.
\section{$S^1$-quotient of the Spinor bundle of $S^4$}
Let us first outline our general strategy to performing the quotient construction. Recall that the fibres of the spinor bundle of $S^4$ are diffeomorphic to $\R^4\simeq \C^2$. We shall consider the action of the diagonal $U(1)$ in $SU(2)$ on the fibres. This fibrewise quotient can be interpreted as a reverse Gibbons-Hawking (GH) ansatz. We begin by giving a brief overview of the GH construction in subsection \ref{bsquotient} and describe it in detail for the Hopf map by viewing our quotient construction as a fibrewise Hopf fibration in subsection \ref{fibrequotient}. Extending this to the total space we construct the quotient $G_2$-structure on the anti-self dual bundle of $S^4$, see subsection \ref{bsconequotient}. From the results of section \ref{tf} we know that the quotient $G_2$-structure cannot be torsion free but on the other hand, it is well-known that the anti-self dual bundle of $S^4$ also admits a holonomy $G_2$ metric cf. \cite{Bryant1989}. Motivated by the fact that both of these $G_2$-structures are asymptotic to a cone metric on $\C \mathbb{P}^3$ we study the induced $SU(3)$-structures. In subsection \ref{remarksonsu(3)} we give explicit formulae for the $SU(3)$-structures on the link and show that in both cases the induced almost complex structure corresponds to the Nearly-K\"ahler one.

\subsection{The Gibbons-Hawking ansatz}\label{bsquotient}
Since we shall use the Gibbons-Hawking ansatz in the next section, we quickly describe the general construction. In essence it provides a local construction of hyperK\"ahler metrics starting from a $3$-manifold together with a connection form on an $S^1$-bundle and a harmonic function. We begin by recalling the definition of a hyperK\"ahler manifold. 
\begin{Def}
An oriented Riemannian manifold $(M^{4n},g)$ is called hyperK\"ahler (HK) if it admits a triple of closed non-degenerate $2$-forms $\om_1$, $\om_2$ and $\om_3$ satisfying the compatibility conditions
\begin{align*}
\frac{1}{2}\ \om_i \w \om_j=\delta_{ij}\ dvol_g.
\end{align*}
\end{Def}
Let $U$ be an open subset of $\R^{3}$ with the standard Euclidean metric $g_{0}$ and $M^{4}$ a principal $S^{1}$ bundle on $U$ generated by a vector field $X$ normalised to have period $2\pi.$ Suppose we are also given a connection $1$-form $\eta$ on $M^{4}$ such that $\eta(X)=1$ (using the natural identification $\mathfrak{u}(1)\cong \R$). For a positive harmonic function $f$ on $U$ satisfying $*_{g_0}df=d\eta$, the metric 
\[g_{M^{4}}= f \pi^{*}g_{0}+f^{-1}\eta\otimes \eta,\]
and the anti-self-dual (ASD) $2$-forms
\begin{align*}
\omega_{1}= \eta \wedge dx_{1} - f dx_{2}\wedge dx_{3}\\
\omega_{2}= \eta \wedge dx_{2} - f dx_{3}\wedge dx_{1}\\
\omega_{3}= \eta \wedge dx_{3} - f dx_{1}\wedge dx_{2}
\end{align*}
define a HK structure on $M^{4}.$ By construction the triple of symplectic forms are closed:
\begin{align*}
d\omega_{1}&=(*df)\wedge dx_{1} -df \w dx_{2}\w dx_{3}\\&=\partial_{1}f dx_{2}\w dx_{3} \w dx_{1}-\partial_{1}f dx_{1}\w dx_{2} \w dx_{3}\\&=0
\end{align*}
and likewise for $\omega_{2}$ and $\omega_{3}$. The compatible almost complex structures are defined by $g(J_{i}v,w)=\omega_{i}(v,w).$ The closedness of $\omega_{i}$ is equivalent to $\nabla^{g_{M^4}} J_{i}=0$ i.e. $J_i$ are indeed complex structures and thus $Hol(g)\subseteq Sp(1).$

Note that setting $U=\R^{3}$, $\alpha=dx_{0}$ and $f$ constant gives a flat HK cylinder. More interestingly, the projection map $\pi: \R^4 - \{0\} \to \R^{3}-\{0\}$ given in quaternionic coordinates by $\pi(p)=\frac{1}{2}\bar{p}i p$ is the moment map of the Hopf bundle, where the $S^1$ action is generated by left multiplication by $-i$. It turns out that this map can be smoothly extended to the origin whenever $f$ is a suitable harmonic function. Moreover one can recover the flat HK metric on $\R^4$, which we shall describe explicitly in the next subsection.
\subsection{$S^1$-quotient of a fixed fibre of the spinor bundle}\label{fibrequotient}
We begin by reminding the reader of the construction of the Bryant-Salamon $Spin(7)$ manifold.
Given $S^{4}$ with the standard round metric and orientation, we denote by $P\simeq SO(5)$ the total space of the $SO(4)$-structure. Since $H^{2}(S^{4},\R)=0$, in particular the second Stiefel-Whitney class vanishes hence it is a spin manifold so we can lift $P$ to its double cover $\tilde{P}$. The spin group can be described explicitly via the well-known isomorphism
$$Spin(4)\cong Sp(1)_{+}\times Sp(1)_{-} \cong SU(2)_{+}\times SU(2)_{-}$$
where the $\pm$ subscripts distinguish the two copies of $SU(2)$. 
Taking the standard representation of $SU(2)_{-}$ on $\C^{2}_-$, we construct the (negative) spinor bundle $V_{-}:=\tilde{P} \times_{SU(2)_{-}} \C^{2}_-$ as an associated bundle. 

There is also an action of $SU(2)$ on the fibres of $V_{-}$ which can be described as follows. If we ignore the complex structure the fibres of $V_{-}$ are simply $\R^4$ and its complexification is isomorphic to $\C^2_{-}\otimes \C^{2}.$ The desired $SU(2)$ action is the standard action on $\C^2$ and is well-defined on the realification of $V_{-}\otimes \C$. In the description of the Bryant-Salamon construction in \cite{Bryant1989}, this action on the fibre can also be viewed as a global $Sp(1)$ action (acting on the right) on $\Q$ in
\[\tilde{P} \times \Q \xrightarrow{/Sp(1)_-} V_-, \]
thus commuting with the left action of $Sp(1)_{-}$ and hence passes to the quotient. Having now justified the existence of this $SU(2)$ action, we fix a point, $p \in S^{4}$ and describe the action of an $S^{1}\hookrightarrow SU(2)$ on the fibre of $V_{-}.$ This will enable us to describe a fibrewise
HK quotient and then reconstruct the $\R^4$ fibre using the Gibbons-Hawking ansatz with harmonic function $f=1/2R$ where $R$ denotes the radius in $\R^3-\{0\}$ as described in the previous section. Note that topologically the base manifold is just the anti-self dual bundle of $S^4 $ which we denote by $\Lm^2_-S^4$. This is due to the fact that the quotient construction reduces the $Sp(1)_-$ action on the $\R_-^4$ fibre to an action of $SO(3)_-$ on $\R^3=\R^4/S^1$, as we shall see below, and the associated bundle construction for this representation is $\Lm^2_- S^4$ cf. \cite{Salamon1989}.

Let $(x_{1},x_{2},x_{3},x_{4})$ denote the coordinates on the fibre, so that we may write the fibre metric as
\[g = \sum_{i=1}^4 dx_i\otimes dx_i\]
i.e. $g$ denotes the restriction of the Bryant-Salamon metric $g_{\Phi}$ to the vertical space.
Denoting by $r$ the radius function in the fibre, i.e. $r^{2}=\sum_{i=1}^4 x_{i}^{2}$, we have $rdr=\sum_{i=1}^4 x_i dx_{i}$. We make the identifications $\R^4\cong\C^2\cong\mathbb{H}$ by
\[(x_1,x_2,x_3,x_4)=(x_1+ix_2,x_3+ix_4)=x_1+ix_2+jx_3-kx_4\]
Consider now the $U(1)$ action on $\R^{4}\cong \C^{2}$ given by $$e^{i\theta}(z_{1},z_{2})=(e^{-i\theta}z_{1},e^{+i\theta}z_{2})$$ or equivalently by left multiplication by $-i$ on $\Q$. Note that this $S^1$ is just the diagonal torus of $SU(2)$.
The Killing vector field $X$ generating this action is given by 
$$ X = x_2\pd{x_1} -x_1\pd{x_2} -x_4\pd{x_3} +x_3\pd{x_4}.$$
and thus $\|X\|_{g}=r$. We also endow the fibre with a HK structure given by the triple
\[\begin{array}{rcl}
\gamma_1 &=& dx_1\wedge dx_2 - dx_3\wedge (-dx_4)\y
\gamma_2 &=& dx_1\wedge dx_3 - (-dx_4)\wedge dx_2\y
\gamma_3 &=& dx_1\wedge (-dx_4) - dx_2\wedge dx_3
\end{array}\]

They can be extended to a local
orthonormal basis of the bundle $\Lm^2_-S^4$ but the resulting
forms will \emph{not} be closed. The spin bundle \emph{does} have a
global HK structure, but arising from $SU(2)_+$ and since we have already 
fixed one of its complex structures, this HK structure is not relevant. 
In view of our quotient construction, we define 
$$
\eta :=r^{-2} g_\Phi(X, \cdot)=r^{-2}( x_{2}dx_1 -x_{1}dx_2-x_{4}dx_3+x_{3}dx_4) 
$$
i.e $\eta$ is a connection $1$-form on $V_-$.%\eta=r^{-2}\theta 
\ The map
\begin{align*}
\mu \colon\ \R^4 &\to \R^3 \\
(x_1,x_2,x_3,x_4) &\mapsto (\mu_1,\mu_2,\mu_3)
\end{align*}
where
\[\begin{array}{rcl}
\mu_1 &=& \frac{1}{2}(x_1^2 +x_2^2 -x_3^2 -x_4^2) \y
\mu_2 &=& x_1x_4 +x_2x_3 \y
\mu_3 &=& x_1x_3 -x_2x_4.
\end{array}\]
is the HK moment map for the $U(1)$
action. By identifying $\R^3$ with $\text{Im}(\Q)$, $\mu$ can also be expressed using quaternions as:
\[ \mu(q)= \frac{1}{2}\overline q\,i\,q,\qquad q=x_1+x_2i+x_3j-x_4k.\]
making the $S^1$-invariance clear.
Thus $\mu$ induces
a diffeomorphism 
\[\R^4/U(1)\simeq \R^3.\]
Note that strictly speaking this action is not free but nonetheless the construction can be carried out on $\R^4-\{0\}$ and can be extended smoothly to the origin.
A direct computation gives
$$\begin{aligned}
\gamma_{3} &= dx_{32}+dx_{41}\\
&=r^{-2}\Big((x_2 dx_1 -x_1dx_2 -x_4dx_3+x_3dx_4)\w(x_1dx_3+x_3dx_1-x_2dx_4-x_4dx_2)\\
&\ \ \ -(x_1 dx_1 +x_2 dx_2+x_3 dx_3+x_4dx_4)\w(x_1dx_4+x_4dx_1+x_2dx_3+x_3dx_2)\Big)\\ 
&=\eta\wedge d\mu_3-f\,d\mu_1\wedge d\mu_2.
\end{aligned}$$
where $f=\frac{1}{2 R}$ and $R=\sqrt{\mu_1^2+\mu_3^2+\mu_3^2}$ is the radius on $\R^3$. Similarly we obtain
\[\begin{array}{rcl} 
\gamma_1 &=& \eta\wedge d\mu_1-f\,d\mu_2\wedge d\mu_3\y
\gamma_2 &=& \eta\wedge d\mu_2-f\,d\mu_3\wedge d\mu_1
\end{array}\]
This confirms that $\eta$ is the connection form that
features in the Gibbons-Hawking ansatz with
\[  g_{\R^4} = f^{-1}\eta\otimes\eta + f\,\pi^*g_U,\]
where $g_U=d\mu_1^2+d\mu_2^2+d\mu_3^2$ is the Euclidean metric on $\R^3$ with volume form $vol_{\R^3}=d\mu_{123}$. Using $R^2=\sum_{i=1}^3\mu_i^2=\frac{1}{4}r^4$ we can directly verify that
\[
*_{\R^3}df=d\eta
.\]

Having described the GH ansatz for the Euclidean space we proceed to our main example.
\subsection{$S^1$-quotient of the Bryant-Salamon cone metric}\label{bsconequotient}
We shall now take the quotient of Bryant-Salamon metric by applying the above construction to each $\R^4$ fibre.
The conical Bryant-Salamon $Spin(7)$ $4$-form is given (pointwise) in our notation by
\[ \Phi =16 r^{-8/5} dx_{1234} +20 r^{2/5}\sum \gamma_i\wedge\epsilon_i
+25r^{12/5}dvol_{S^4},\]
where $\{\epsilon_i\}$ is a local basis of ASD forms on $S^4$  and $dvol_{S^4}$ is the (pullback of) the volume form.
The $\Spin(7)$ metric is then given by
\[ g_{\Phi}=4r^{-4/5}\sum_{i=1}^4{dx_i}^2+5r^{6/5}g_{s^{4}},\]
and so the $1$-forms $d\mu_i$ (or rather, $\pi^*d\mu_i=d(\mu_i\circ\pi)$) have
norm
\[ \|d\mu_i\|_{\Phi}^2=\frac{1}{4}r^{14/5}.\]
On the other hand, from (\ref{basicequation}) we compute
\[  s^{-2}=g_\Phi(X,X)=4r^{6/5},\]
so $s=r^{-3/5}$. 
We know that the $G_2$ metric $g_{\vp}$ satisfies
\[ g_{\vp} = s^{-2/3}(g_{\Phi}-s^{-2}\eta^2)=r^{2/5}(g_{\Phi}-4r^{6/5}\eta^2).\]
Considering the volume form 
of the fibre of the quotient we have
\[\begin{array}{lcl}-r^{-2} d\mu_{123}
&=&-x_3dx_{123}-x_4dx_{124}+x_1dx_{134}+x_2dx_{234}\\
&=& X\ip dx_{1234}.
\end{array}\]
Defining $d\nu_i=\iota_{X}\gamma_{i}$ we have that $d\nu_{123}=-d\mu_{123}.$
Putting all together we have
\[\begin{array}{rcl} X\ip\Phi 
&=& X\ip\,16 r^{-8/5}dx_{1234} +20 r^{2/5}\sum(X\ip\gamma_i)\wedge\epsilon_i\y
&=& 2^{11/5}(R^{-9/5}d\nu_{123} + 5R^{1/5}\sum_{i=1}^3 d\nu_i\wedge\epsilon_i).
\end{array}\]
We can now extend this pointwise construction to the whole of $V^-.$
From our construction, the induced $G_{2}$-structure on the quotient is given (after rescaling) by
\[ \vp^{}_{GH}\ =\ \frac{1}{6}R^{-9/5}\,\beta + 5R^{1/5}\,d\tau\]
We are here using the globally well-defined forms defined in \cite[pg 94]{Salamon1989} (see also the appendix below) where $\tau$ is tautological $2$-form on the ASD bundle and $\frac{1}{6}\beta$ is the volume form of the fibre which was pointwise denoted by $dx_{1234}$. By contrast the holonomy $G_{2}$ form is given by 
\[ \vp^{}_{BS}\ =\ \frac{1}{6}R^{-3/2}\,\beta +2 R^{1/2}\,d\tau.\]
Since the Bryant-Salamon metric on $\R^+\times\C \mathbb{P}^3$ is just the cone metric on $\C\mathbb{P}^3$ endowed with its Nearly-K\"ahler (NK) structure we may also write it as
\[g_{BS}=dt^2+t^2\ (\frac{1}{2}g_{S^4}+\frac{1}{4}\hat{g}_{S^2} )\]
where $t$ denotes the coordinate of $\R^+$ and $g_{NK}:=\frac{1}{2}g_{S^4}+\frac{1}{4}\hat{g}_{S^2}$ is the NK metric (up to homothety). Here we are interpreting $g_{NK}$ as a metric on the twistor space of $S^4$ where $g_{S^4}$ denotes the pullback of the round metric and $\hat{g}_{S^2}$ the metric on the $S^2$ fibres (see the appendix for more details). Comparing $\vp_{BS}$ with $\vp_{GH}$ and using our expression for $g_{BS}$ we can perform a pointwise computation as above and show that
\[g_{GH}=dt^2 +\frac{8}{5}t^2\ (\frac{1}{2}g_{S^4}+\frac{1}{10}\hat{g}_{S^2}).\]
The quotient metric is thus the cone metric on the twistor space of $S^4$ but with ``smaller'' $S^2$ fibres. In order to gain better understanding of the geometric structure on the $\C \mathbb{P}^3$ we look at the induced $SU(3)$-structure.
\subsection{Remarks on the induced $SU(3)$-structure on $\C \mathbb{P}^3$}\label{remarksonsu(3)}
We remind the reader that an $SU(3)$-structure on a $6$-manifold consists of a non-degenerate $2$-form $\om$ and a pair of $3$-forms $\Om^\pm$ satisfying the compatibility conditions
\[ \om \w \Om^\pm=0 \text{\ \ and \ \ }\frac{2}{3} \om^3={\Om^+ \w \Om^-} .\] 
The relevance of this here comes the fact that oriented hypersurfaces in $G_2$-structures naturally inherit such a structure. If $\mathbf{n}$ denotes the unit normal to a hypersurface $Q^6$ then the forms are given by:
\begin{align*}
\om&=\mathbf{n}\ip\vp\\
\Om^+&=\vp\big|_{Q^6}\\
\Om^-&=-\mathbf{n}\ip*_\vp\vp. 
\end{align*}
It is known that the NK structure on $\C P^3$ satisfies
\[ d\om_{NK} =3\ \Om^+_{NK} \text{\ \ and  \ \ } d\Om_{NK}^-=-2\ \om_{NK}^2. \]
In contrast the $SU(3)$-structure $(\om_{GH},\Om^+_{GH},\Om^-_{GH})$ on the link (for $t=1$) of the quotient $G_2$-structure satisfies  
\begin{gather*}
d\om_{GH}=3\ \Om^+_{GH}\\
d\Om^-_{GH}=-2\ \om_{GH}^2-\frac{1}{5}\ (\frac{1}{5}\ \sigma -\tau)\w \om_{GH}
\end{gather*}
The proof is a straightforward computation using the formulae in the appendix. Two things worth noting are that $\Om_{GH}^+=\frac{32}{25}\ \Om^+_{NK}=\frac{8}{25}\ d\tau$ so in particular both define the same almost complex structure and the extra-torsion component $\frac{1}{5}\sigma-\tau$ lives in $[\Lm^{1,1}_0].$ Using the formulae from \cite[Thm $3.4$-$3.6$]{Bedulli2007} we can confirm directly that this metric is not Einstein which is consistent with the canonical variation approach \cite[pg. 258]{Besse2008} which asserts that there are only two Einstein metrics in this family, the Fubini-Study metric and the NK one. Moreover it was also shown in \cite{Foscolo2015a} that in fact there are no other cohomogeneity one NK structure on $\C P^3$. Nonetheless the scalar curvature of $g_{GH}$ is still constant and positive:
\begin{align*}
Scal(g_{GH})&=30-\frac{1}{2}\cdot\|\frac{1}{5}(\frac{1}{5} \sigma-\tau) \|^2_{g_{GH}}\\
&=30-\frac{1}{2}\cdot\frac{3}{8}\\
&=\frac{477}{16}>0.
\end{align*} 
It is also worth pointing out that this $SU(3)$-structure is half-flat cf. \cite{ChiossiSalamonIntrinsicTorsion, Fino2015} and as such can be evolved by the Hitchin flow to construct a torsion free $G_2$-structure. The resulting metric belongs to the general class of metrics of the form
\[g=dt^2+a(t)^2\hat{g}_{S^2}+b(t)^2g_{S^4} ,\]
which were considered in \cite[Sect. $5$B]{GibbonsCoho1}. It was also shown, after suitable normalisations, that the Bryant-Salamon metrics are the only solutions to this system. 

\begin{Rem}
Observe that, as in the GH ansatz for the Hopf map, this construction extends to the smooth Bryant-Salamon $Spin(7)$ metric with the same circle action but which now has as fixed point locus an $S^4$ corresponding to the zero section of the spinor bundle. Extending the above construction to the smooth metric simply amounts to replacing $R$ by $R+1$ in the expressions $\vp_{BS}$ and $\vp_{GH}$. Thus, we obtain a closed $G_2$-structure on all of $\Lm^2_-S^4$. 
\end{Rem}

\section{$S^1$-quotient of flat $Spin(7)$ metric}\label{flatquotient}

We now consider a simpler situation: that of the $S^1$-reduction of the flat $Spin(7)$ structure $\Phi_0= \frac{1}{8}(-d \al_1^2 + d \al_2^2 + d \al_3^2)$ on $\R^8$ where
\begin{align*}
\al_1&= -x_1dx_0+x_0dx_1+x_3dx_2-x_2dx_3-x_5dx_4+x_5dx_4+x_7dx_6-x_6dx_7,\\
\al_2&= -x_2dx_0+x_0dx_2+x_1dx_3-x_3dx_1-x_6dx_4+x_4dx_6+x_5dx_7-x_7dx_5,\\
\al_3&= -x_3dx_0+x_0dx_3+x_2dx_1-x_1dx_2-x_7dx_4+x_4dx_7+x_6dx_5-x_5dx_6.
\end{align*}
This explicit construction was motivated by the work of Acharya, Bryant and Salamon \cite{Acharya2019circle} where they investigate the $S^1$-reduction of the conical $G_2$ metric on $\R^+ \times \C P^3.$ We can identify $\R^8$ with coordinates $(x_0,x_1,...,x_7)$ with $\Q^2$ by $(x_0+ix_1+jx_2+kx_3,x_4+ix_5+jx_6+kx_7).$ There are natural actions given by $Sp(2)$ acting by left multiplication and $Sp(1)$ acting by multiplication on the right. The $1$-forms $\al_i$ are simply the dual of the $S^1$ actions given by right multiplication by the imaginary quaternions. We consider the $S^1$ action generated by the vector field
\[X=-x_1 \partial_0+ x_0 \partial_1 -x_3 \partial_2+ x_2 \partial_3 -x_5 \partial_4+ x_4 \partial_5 -x_7 \partial_6+ x_6 \partial_7 \]
given by a diagonal $U(1) \subset Sp(2).$ A simple computation shows that
\[ d ( X \ip d\al_i \w d\al_i)=0 \text{ \ \ for\ \ } i=1,2,3 \]
from which it follows that $\mathcal{L}_X \Phi_0=0$. Thus we get a closed $G_2$-structure on the quotient space $\R^+ \times \C P^3$ given by $\vp= \iota_X \Phi$ from \ref{basicequation}.  Noting that $\Phi_0$ is also invariant by the right $S^1$ action generated by the vector field
\[Y=-x_1 \partial_0+ x_0 \partial_1 +x_3 \partial_2- x_2 \partial_3 -x_5 \partial_4+ x_4 \partial_5 +x_7 \partial_6- x_6 \partial_7 \]
i.e $\mathcal{L}_Y\Phi_0=0$ and that both $S^1$ actions commute, we can take the (topological) $\mathbb{T}^2$ reduction to the $6$-manifold $\R^3 \oplus \R^3 - \{0\}$. More concretely, we can split $\R^8=\R^4 \oplus \R^4$ with coordinates $x_0,x_1,x_4,x_5$ on the first factor and $x_2,x_3,x_6,x_7$ on the second and we consider the equivalent $\mathbb{T}^2$ action given by the vector fields $\frac{1}{2}(X+Y)$ and $\frac{1}{2}(X-Y)$, each acting non-trivially on only one $\R^4$ factor. Using the HK moment maps as in the previous section we get coordinates $u_i$ and $v_i$ on $\R^3 \oplus \R^3-\{0\}$ given by
\begin{align*}
u_1&=x_0^2+x_1^2-x_4^2-x_5^2    &  v_1&=x_2^2+x_3^2-x_6^2-x_7^2\\
u_2&=2\ (x_0 x_4+x_1 x_5)           &
v_2&=2\ (x_2 x_6+x_3 x_7)\\
u_3&=2\ (x_0 x_5 - x_1 x_4)           &
v_3&=2\ (x_2 x_7 - x_3 x_6).
\end{align*}
These coordinates can now be pulled back to $\R^+ \times \C P^3$ and will allow us to give an explicit expression for $\vp$. From this point of view we have the $S^1$-bundle: $$\R^+\times \C P^3 \xrightarrow{/S^1} \R^3\oplus\R^3-\{0\}.$$ 
Following the Apostolov-Salamon construction \cite{Apostolov2003} we can write 
\begin{align}
\label{AS1}\vp&=\xi \w\om + H^{3/2}\ \Om^+\\
*_\vp\vp&=\frac{1}{2}\ H^2 \om^2 - \xi\w H^{1/2} \ \Om^-
\end{align}
where $H:=\|Y\|^{-1}_\vp$, $\xi$ is the connection $1$-form defined by
\[ \xi(\cdot):=H^2\ g_\vp(Y,\cdot) \]
and $(\om,\Om^+,\Om^-)$ is the $SU(3)$-structure induced on $\R^3\oplus\R^3-\{0\}$. We now give coordinate expressions for the aforementioned differential forms.%strictly speaking we mean the pullback of g_\vp to \R^8 in the definition of \xi. the fact that this is indeed well-defined is since X and Y commute i.e. [X,Y]=0 (this was checked using maple)
\begin{Prop}\label{quotient}
In the above notation the closed $G_2$-structure on $\R^+ \times \C P^3$ given by $\vp = \iota_X \Phi_0$ can be expressed as
\[ \vp= \xi \w \frac{1}{2}\sum_{i=1}^{3}  dv_i \w du_i  +\frac{1}{8}\bigg( \frac{1}{u}\big(du_{123}-\{d \boldsymbol{v}, d \boldsymbol{u},d \boldsymbol{u}\}\big) + \frac{1}{v}\big(dv_{123}-\{d \boldsymbol{v}, d \boldsymbol{v},d \boldsymbol{u}\}\big)\bigg) ,\]	where $\{d \boldsymbol{v}, d \boldsymbol{v},d \boldsymbol{u}\}$ denotes $$dv_1 \w dv_2 \w du_3 + dv_2 \w dv_3 \w du_1+ dv_3 \w dv_1 \w du_2,$$ similarly for $\{d \boldsymbol{v}, d \boldsymbol{u},d \boldsymbol{u}\}$. 
Moreover we have
\[H^{\frac{1}{2}}\Om^- = \frac{1}{4 R^{\frac{2}{3}}}\bigg(\{ d\boldsymbol{v},d\boldsymbol{v},d\boldsymbol{u} \} - \{d\boldsymbol{u},d\boldsymbol{u},d\boldsymbol{v} \} + \frac{u}{v} dv_{123} - \frac{v}{u} du_{123} \bigg), \]
\[H=\frac{R^{2/3}}{2\ u^{1/2} \ v^{1/2} } ,\]
where $R^2:=x_0^2+\cdots+x_7^2$, $u^2:=u_1^2+u_2^2+u_3^2=(x_0^2+x_1^2+x_4^2+x_5^2)^2$ and likewise for $v.$ %The metric induced by $(\om,\Om^+)$ on $\R^3\oplus\R^3-\{0\}$ is given by
%\[g_\om=\frac{1}{2} \bigg( \frac{v^{1/2}}{u^{1/2}}\ (du_1^2+du_2^2+du_3^2) + \frac{u^{1/2}}{v^{1/2}}\ (dv_1^2+dv_2^2+dv_3^2) \bigg).\]
%and 
The curvature of the $S^1$-bundle over $\R^3 \oplus \R^3 -\{0\}$ is given by
\[d \xi= - \frac{\{\boldsymbol{v},d\boldsymbol{v},d\boldsymbol{v}\}}{4 (v_1^2+v_2^2+v_3^2)^{3/2}} + \frac{\{\boldsymbol{u},d\boldsymbol{u},d\boldsymbol{u}\}}{4 (u_1^2+u_2^2+u_3^2)^{3/2}},\]
where $\{\boldsymbol{v},d\boldsymbol{v},d\boldsymbol{v}\}$ denotes
$$ v_1 dv_2 \w dv_3+v_2 dv_3 \w dv_1+v_3 dv_1 \w dv_2,$$
and likewise for $\{\boldsymbol{u},d\boldsymbol{u},d\boldsymbol{u}\}$. %Using $g_\om$ and $\om$ we find that the induced almost complex structure $J$ is given by:
%\[J(u^{1/2} \partial_{u_i})=v^{1/2}\partial_{v_i} , \text{\ \ \  for\ \  }i=1,2,3. \]
%By pulling back $\xi$ to $\R^8$ we can write it as
%\[ \xi=\frac{x_0 dx_1-x_1dx_0+x_4 dx_5-x_5 dx_4}{2 u} + \frac{-x_2 dx_3+x_3dx_2-x_6 dx_7+x_7 dx_6}{2 v}.\]
\end{Prop}
The proof is a long computation which was carried out with the help of \textsc{Maple}. One can directly verify the above formulae hold using the definitions of $u_i$, $v_i$ and expressing them in terms of $x_i.$ The reader might find it interesting to compare our expressions to those in \cite{Acharya2019circle} for the torsion free $G_2$ quotient.
%Note that since $s=R^{-1}$ is not constant and $\vp$ is not torsion free, $\eta$ is neither an instanton nor an anti-instanton.

In \cite{Hitchin2000} Hitchin shows that an $SU(3)$-structure is completely determined by the pair $(\om,\Om^+)$. Note that here $\Om^+$ can easily be read off from the expressions for $\vp$ and $H$ in Proposition \ref{quotient} and formula (\ref{AS1}). Thus, we can explicitly compute the induced complex structure and metric on $\R^3\oplus\R^3-\{0\}$.
\begin{Prop}
The metric induced by $(\om,\Om^+)$ on $\R^3\oplus\R^3-\{0\}$ is given by
\[g_\om=\frac{1}{2} \bigg( \frac{v^{1/2}}{u^{1/2}}\ (du_1^2+du_2^2+du_3^2) + \frac{u^{1/2}}{v^{1/2}}\ (dv_1^2+dv_2^2+dv_3^2) \bigg)\]
and the almost complex structure $J$ by
\[J(u^{1/2} \partial_{u_i})=v^{1/2}\partial_{v_i} , \text{\ \ \  for\ \  }i=1,2,3. \]
\end{Prop}
\noindent Note that since $\vp$ is closed from (\ref{g2torsion2}) we have that
\[ d*_\vp \vp = \tau_2 \w \vp. \]
We shall now derive an explicit expression for the torsion of the $G_2$-structure $\vp.$ 
Under the inclusion $SU(3)\hookrightarrow G_2$ we can write the torsion form as
$$\tau_2 = \xi \w \tau_v + \tau_h $$
where $\tau_v$ and $\tau_h$ are basic $1$-form and $2$-form respectively i.e. they are (pullback of) forms on $\R^3\oplus\R^3-\{0\}$. It is not hard to show that $\tau_h \in [\Lm^{2,0}] \oplus [[\Lm^{1,1}_0]]$ and that the $[\Lm^{2,0}]$-component of $\tau_h$ is $SU(3)$-equivalent to $\tau_v.$ We compute $\tau_h$ and $\tau_v$ as
\begin{align*}
\tau_h \cdot (3 u v\cdot R^{8/3})=&-u\cdot(\frac{1}{2}(\{\boldsymbol{u}, d \boldsymbol{v},d \boldsymbol{v}\}+\{\boldsymbol{v}, d \boldsymbol{v},d \boldsymbol{v}\})+\frac{3u}{2v} \{\boldsymbol{v}, d \boldsymbol{v},d \boldsymbol{v}\})\\
&-v\cdot(\frac{1}{2}(\{\boldsymbol{v}, d \boldsymbol{u},d \boldsymbol{u}\}+\{\boldsymbol{u}, d \boldsymbol{u},d \boldsymbol{u}\})+\frac{3v}{2u} \{\boldsymbol{u}, d \boldsymbol{u},d \boldsymbol{u}\})\\
&-\frac{1}{2}( u \{\boldsymbol{v}, d \boldsymbol{v},d \boldsymbol{u}\}+ v\{\boldsymbol{u}, d \boldsymbol{v},d \boldsymbol{u}\})\\
&-\frac{1}{2}( v \{\boldsymbol{u}, d \boldsymbol{u},d \boldsymbol{v}\}+ u\{\boldsymbol{v}, d \boldsymbol{u},d \boldsymbol{v}\})
\end{align*}
and 
\begin{align*}
\tau_v\cdot (\frac{3}{2}\cdot R^{8/3})=\ & \sum_{i=1}^3\bigg(\frac{1}{  v}(v u_i - 3 u v_i)\ dv_i -\frac{1}{u}(u v_i - 3 v u_i)\ du_i   \bigg)\\
=\ & \bigg( \boldsymbol{u}\cdot d\boldsymbol{v}-\boldsymbol{v}\cdot d\boldsymbol{u}-3(u dv -vdu) \bigg)
\end{align*} 
where $\boldsymbol{u}\cdot d\boldsymbol{v}$ denotes $\sum_{i=1}^{3}u_i dv_i$ and likewise for $\boldsymbol{v}\cdot d\boldsymbol{u}$. From these expressions one can show that the $[\Lm^{1,1}_0]$-component of $\tau_h$ is non zero i.e. $J$ is non-integrable.
%Note that since $s=R^{-1}$ is not constant and $\vp$ is not torsion free, $\eta$ is neither an instanton nor an anti-instanton.

\begin{Rem}\
\begin{itemize}

\item If we restrict the $Spin(7)$ $4$-form $\Phi_0$ on $\R^8$ to $S^7$ we get a $G_2$ $4$-form $*_{S^7}\vp_{S^7}$ and the flat metric restricts to give the standard round metric. Since the cone metric is just the flat metric again, this means that this cocalibrated $G_2$ structure is inducing the round metric. This statement is in agreement with the fact that with the round metric $S^7$ is a $3$-Sasakian manifold. Note that in contrast the squashed Einstein metric on $S^7$ has exactly one Killing spinor so the cone metric has holonomy \textit{equal} to $Spin(7)$ \cite{FRIEDRICH1997259, Friedrich1990}. We can now take the $S^1$-quotient with respect to any free $S^1$ action preserving the round nearly parallel $G_2$-structure. Since this quotient is also a Riemannian submersion (as the size of the circle orbits are constant) the quotient metric is just the Fubini-Study metric. However by contracting the $4$-form with the vector field generated by the $S^1$ we get the (negative) imaginary part of a $(3,0)$ form on the $\C P^3$. The latter induces an almost complex structure compatible with the Fubini-Study metric but which definitely cannot be the integrable one, otherwise this contradicts the fact that the canonical bundle of $\C P^3$ with the Fubini-Study complex structure is non-trivial. The above closed $G_2$-structure is then just the Riemannian cone on this $\C P^3$. More explicitly, we can write the flat metric on $\R^8$ as
\[g_{\R^8}=dR^2+R^2(\eta^2+g_{FS})=R^2\eta^2+R^{-2/3}g_\vp \]
where $\eta$ is just the connection form of the $S^1$ action for the Fubini-Study quotient as above and $s=\|X\|^{-1}_\Phi=R^{-1}.$ Thus the metrics of proposition \ref{quotient} can also be expressed as
\begin{gather*}
g_\vp=R^{2/3}\cdot dR^2+R^{8/3}g_{FS}=dr^2+\frac{16}{9}r^2 g_{FS} \\
g_\om= 2(u\cdot v)^{1/2}\ (dR^2 +R^2\ g_{FS}-\frac{4\ u\cdot v}{R^2}\xi^2) .
\end{gather*}
Note that by construction, the latter metric is invariant under the vector field $Y$ and thus, passes to the quotient $(\R^+\times \mathbb{C}P^3)/S^1$.
\item Observe that one can also view this construction as a $\mathbb{T}^2$-quotient of a $Spin(7)$-structure to a $6$-manifold endowed with an $SU(3)$-structure $(\om,\Om^+,\Om^-)$ given by
\[\Phi=\eta \w \xi \w \om + H^{3/2}\eta \w \Om^+ + \frac{1}{2}s^{4/3}H^2 \om \w \om - s^{4/3}H^{1/2}\xi \w \Om^-  \]
and the metrics are related by
\[ g_\Phi = s^{-2}\eta^2 + s^{2/3}H^{-2} \xi^2 + s^{2/3}Hg_\om. \]
This quotient construction under the assumption that the six-manifold is K\"ahler is currently work in progress by the author.
\end{itemize}
\end{Rem}

\bibliography{referencing2}
\bibliographystyle{plain}

\section{Appendix}
For the convenience of the reader and to make this article self-contained we describe the construction of the Bryant-Salamon metrics on the anti-self dual bundle of $S^4$. We shall follow the approach described in \cite{Salamon1989}. The reader will find proofs of the assertions made therein.

Consider $S^4$ embedded in $\R^5$ with coordinates $x_1,...,x_5$ we may choose the following local orthonormal frame
$$
v_1 = \frac{1}{R}
\begin{pmatrix}
x_2  \\
-x_1 \\
x_4 \\
-x_3  \\
0    
\end{pmatrix},\
v_2 = \frac{1}{{R}}
\begin{pmatrix}
-x_3  \\
x_4 \\
x_1 \\
-x_2  \\
0    
\end{pmatrix},\
v_3 = \frac{1}{R}
\begin{pmatrix}
x_4  \\
x_3 \\
x_2 \\
-x_1  \\
0    
\end{pmatrix},\
v_4 = \frac{1}{\sqrt{-1+\frac{1}{x_5^2}}}
\begin{pmatrix}
-x_1  \\
-x_2 \\
-x_3 \\
-x_4  \\
-x_5+\frac{1}{x_5}    
\end{pmatrix},
$$
where $R^2=x_1^2+x_2^2+x_3^2+x_4^2$. Denoting by $e^i$ the corresponding coframe we compute the following
\begin{gather*}
de^1=\frac{2}{R}e^{23}+\frac{\sqrt{1-R^2}}{R}e^{14}\\
de^2=\frac{2}{R}e^{31}+\frac{\sqrt{1-R^2}}{R}e^{24}\\
de^3=\frac{2}{R}e^{12}+\frac{\sqrt{1-R^2}}{R}e^{34}\\
de^4=0
\end{gather*}
In the Cartan moving frame setting the structure equations are given by $d\mathbf{e}=-\om\w \mathbf{e}$ and $F=d\om+\om\w\om \in \Lm^2\otimes \mathfrak{so}(4)$ where $\om$ is the Levi-Civita connection form and $F$ the curvature. We compute them as 
$$\om=-
\begin{pmatrix}
0 & -\frac{1}{R}e^3 & \frac{1}{R}e^2 &\frac{\sqrt{1-R^2}}{R} e^1 \\
\cdot & 0 & -\frac{1}{R}e^1 & \frac{\sqrt{1-R^2}}{R}e^2  \\
\cdot & \cdot & 0 & \frac{\sqrt{1-R^2}}{R}e^3  \\
\cdot & \cdot & \cdot & 0    
\end{pmatrix}\  \text{ and \ }
F=
\begin{pmatrix}
0 & e^{12} & e^{13} & e^{14} \\
\cdot & 0 & e^{23} & e^{24}  \\
\cdot & \cdot & 0 & e^{34}  \\
\cdot & \cdot & \cdot & 0   
\end{pmatrix}
$$
Here we are only writing the upper triangular entries since the matrices are skew-symmetric. The second equation confirms that the round metric has constant curvature and that the scalar curvature is $12$. We can define a local orthonormal basis of the anti-self dual bundle by $c^1:=e^{12}-e^{34}$, $c^2:=e^{13}-e^{42}$ and $c^3:=e^{14}-e^{23}$. $\om$ induces a connection on this bundle given by
$$\nabla c^i=\psi^i_j\otimes c^i$$Since the connection is torsion free we can compute $\psi^i_j$ by
\begin{gather*}
dc^1=\psi^1_2\w c^2+\psi^1_3\w c^3 \\
dc^2=\psi^2_1\w c^1+\psi^2_3\w c^3\\
dc^3=\psi^3_1\w c^1+\psi^3_2\w c^2
\end{gather*}
where $\psi^2_1=\frac{\sqrt{1-R^2}+1}{R}e^1$, $\psi^1_3=\frac{\sqrt{1-R^2}+1}{R}e^2$, $\psi^2_3=\frac{\sqrt{1-R^2}+1}{R}e^3$ and $\psi^i_j=-\psi^j_i.$
These forms can all be pulled back to the total space of the ASD bundle which we denote by the same letter. We introduce fibre coordinates $(a_1,a_2,a_3)$ with respect to the coordinate system defined by $c^i.$ We can define vertical $1$-forms by
\[b^i=da_i+a_j\psi^j_i \]
i.e. they vanish on horizontal vectors. Together with the pull back of the $e^i$ they give an absolute parallelism of the ASD bundle. The following forms are all $SO(4)$-invariant and are hence globally well-defined on the total space:
\begin{gather*}
\rho=a_1a_1+a_2a_2+a_3a_3\\
\sigma=2\ (a_1b^2b^3+a_2b^3b^1+a_3b^1b^2) \\
\alpha=a_1b^2c^3+a_2b^3c^1+a_3b^1c^2-a_1b^3c^2-a_2b^1c^3-a_3b^2c^1\\
\tau=a_1c^1+a_2c^2+a_3c^3\\
\beta=6\ b^{123}.
\end{gather*}
The unit $(\rho=1)$ sphere bundle is diffeomorphic to $\C P^3$ and restricting the above forms we have
\begin{gather*}
g_{FS}=\frac{1}{2}((e^1)^2+(e^2)^2+(e^3)^2+(e^4)^2)+\frac{1}{2}((b^1)^2+(b^2)^2+(b^3)^2)\Big|_{S^2}\\
\om_{FS}=\frac{1}{2}\tau -\frac{1}{4}\sigma \\
g_{NK}=\frac{1}{2}((e^1)^2+(e^2)^2+(e^3)^2+(e^4)^2)+\frac{1}{4}((b^1)^2+(b^2)^2+(b^3)^2)\Big|_{S^2}\\
\om_{NK}=\frac{1}{2}\tau+\frac{1}{8}\sigma\\
\Om_{NK}=\frac{1}{4}(d\tau - i \alpha)
\end{gather*}
The subscript $FS$ refers to the Fubini-Study metric and $NK$ to the Nearly-K\"ahler one. Our choice of scaling was made to fit the conventions of section \ref{bsquotient}. The Bryant-Salamon form is then given by $$\vp_{BS}= u^2 v d\tau + \frac{1}{6}v^3 \beta$$ where 
$u=(2 \rho+1)^{1/4}$ and $v= (2\rho +1)^{-1/4}$.
\end{document}